\documentclass{amsart}
\usepackage{amsfonts,amsmath,amssymb,amsthm}
\usepackage{color}
\usepackage{hyperref}
\usepackage{cite}
\usepackage{hyphenat}

\newtheorem{theorem}{Theorem}
\theoremstyle{plain}

\newtheorem{case}{Case}

\newtheorem{corollary}{Corollary}

\newtheorem{example}{Example}

\newtheorem{proposition}{Proposition}

\numberwithin{equation}{section}

\begin{document}
\title[Conformal vector fields on doubly warped product manifolds]{Conformal
vector fields on doubly warped product manifolds and applications}

\author{H. K. El-Sayied}
\address{Mathematics Department, Faculty of Science, Tanata University, 31527 Tanta, Egypt}
\email{hkelsayied1989@yahoo.com}
\urladdr{http://www.tanta.edu.eg}

\author{Sameh Shenawy}
\address{Modern Academy for engineering and Technology, 11585 Maadi, Egypt}
\email[S. Shenawy]{drssshenawy@eng.modern-academy.edu.eg, drshenawy@mail.com}
\urladdr{http://www.modern-academy.edu.eg}

\author{Noha Syied}
\address{Modern Academy for engineering and Technology, 11585 Maadi, Egypt}
\email[N. Syied]{drnsyied@mail.com}
\urladdr{http://www.modern-academy.edu.eg}

\subjclass[2000]{Primary 53C21; Secondary 53C50, 53C80}
\keywords{Ricci soliton, Killing vector fields, concurrent vector fields,
doubly warped spacetimes.}
\thanks{The authors declare that there is no conflict of interest regarding
the publication of this paper.}

\begin{abstract}
This article aimed to study and explore conformal vector fields on doubly
warped product manifolds as well as on doubly warped space-times. Then we
derive sufficient conditions for matter and Ricci collineations on doubly
warped product manifolds. A special attention is paid to concurrent vector
fields. Finally, Ricci solitons on doubly warped product space-times
admitting conformal vector fields are considered.
\end{abstract}

\maketitle

\section{An introduction}

Bishop and O'Neill introduced Riemannian warped products to construct
manifolds with negative sectional curvature\cite{Bishop1969}. Since then
warped product structures have been widely studied. Doubly warped products
are generalizations of singly warped products. Beem, Ehrilish and Powell
noticed that there are many exact solutions to Einstein's field equation in
the form of warped product manifolds. Since then singly and doubly warped
product manifolds have became more indispensable to physicians and
mathematicians than ever. In \cite{Beem1982}, Beem and Powell studied
Lorentzian doubly warped product manifolds. Allison studied causal
properties, pseudo-covexity and hyperbolicity of doubly warped product
manifolds\cite{Allison1988,Allison1991}. Gebarowski considered doubly warped
products with harmonic Weyl conformal curvature tensor in\cite%
{Gebarowski1993} and conformally flat and conformally recurrent doubly
warped product manifolds in \cite{Gebarowski:1995,Gebarowski1996}. Bulent
Unal studied geodesic completeness of Riemannian and Lorentzian doubly
warped products\cite{Unal:2001}. He also studied hyperbolicity of
generalized Robertson--Walker spacetimes with doubly warped product fibre.
In this paper, Unal finally considered some results about conformal vector
fields of doubly warped products. Doubly warped product submanifolds has
also been studied by many authors in various settings such as M. Faghfouri
and A. Majidi in \cite{Faghfouri:2015}, Olteanu in \cite%
{Olteanu2010,Olteanu2014}, Selcen Perktas and Erol Kilic in \cite%
{Perktas:2010} and many others. Doubly warped spacetimes are good examples
of Lorentzian doubly warped product manifolds.\ These spacetimes are of
interest since they produce many exact solutions to Einstein's field
equations.

In physics, symmetry assumptions are used to understand the relation between
geometry and matter of a spacetime given by Einstein's field equation. For
example, the metric tensor of (pseudo-)Riemannian manifold does not change
under the flow of a Killing vector field i.e. the flow of a Killing vector
field generates a spacetime symmetry. The number of independent Killing
vector fields measures the degree of symmetry of a (pseudo-)Riemannian
manifold. Conformal vector fields have also a well-known geometrical and
physical interpretations and have been studied on (pseudo-)Riemannian
manifolds for a long time. The existence of a conformal vector field on a
spacetime is specially useful to study its geometry. The flow of a conformal
vector consists of conformal transformations of the Riemannian manifold.
Thus, the problems of existence and characterization of different types of
conformal vector fields in different spaces are important and are widely
discussed by both mathematicians and physicists (for example see \cite%
{Berestovskii2008,Deshmokh20141,Deshmokh20142,Kuhnel:1997,Sanchez1999,Steller2006}
and further references contained therein).

The aim of the present paper is to study and explore conformal vector fields
on doubly warped product manifolds as well as doubly warped spacetimes. We
derive many characterizations of conformal vector fields on doubly warped
product manifolds and doubly warped spacetimes. Then we study matter and
Ricci collineation on doubly warped manifolds. One may notice that, after
Pereleman used Ricci soliton to solve the Poincare conjecture posed in 1904,
a growing body of research has continued to study Ricci soliton.
Accordingly, we study Ricci solitons on doubly warped product spacetimes
admitting many types of conformal vector fields. We get some partial answers
of the questions: What do a doubly warped Ricci soliton factors inherit? and
what are conditions under which a doubly warped spacetime is a doubly warped
Ricci soliton?

This article is organized as follows. Section 2 represents some connection
and curvature related formulas on doubly warped product manifolds. In
section 3, we study conformal vector fields on doubly warped product
manifolds. Then we study conformal and concurrent vector fields on doubly
warped spacetimes in two subsections. Finally section 4 a study of Ricci
soliton on doubly warped spacetimes admitting these types of vector fields.
Almost all considerations and statements in this work are local.

\section{Preliminaries}

This section represents connection and curvature related formulas on doubly
warped product manifolds as a generalization of similar results on singly
warped products\cite{Bishop1969,Oneill1983}. Also, we will provide basic
definitions and properties of conformal vector fields.

Let $(M_{i},g_{i},D_{i})$ be two (pseudo-)Riemannian manifolds with metrics $%
g_{i}$ and Levi-Civita connections $D_{i}$ and let $f_{i}:M_{i}\rightarrow
\left( 0,\infty \right) $ be a positive function where $i=1,2$. Also,
suppose that $\pi _{i}:M_{1}\times M_{2}\rightarrow M_{i}$ is the natural
projection map of the Cartesian product $M_{1}\times M_{2}$ onto $M_{i}$
where $i=1,2$. The (pseudo-)Riemannian manifolds doubly warped product
manifold $M=_{f_{2}}M_{1}\times _{f_{1}}M_{2}$ is the product manifold $%
M=M_{1}\times M_{2}$ furnished with the metric tensor%
\begin{equation*}
g=\left( f_{2}\circ \pi _{2}\right) ^{2}\pi _{1}^{\ast }\left( g_{1}\right)
\oplus \left( f_{1}\circ \pi _{1}\right) ^{2}\pi _{2}^{\ast }\left(
g_{2}\right)
\end{equation*}%
where $^{\ast }$ denotes the pull-back operator on tensors. The functions $%
f_{i},$ $i=1,2$ are called the warping functions of the warped product
manifold $M$. In particular, if for example $f_{2}=1$, then $M=M_{1}\times
_{f_{1}}M_{2}$ is called a (singly) warped product manifold. A singly warped
product manifold $M_{1}\times _{f_{1}}M_{2}$ is said to be trivial if the
warping function $f_{1}$ is also constant \cite%
{Agaoka:1998,Faghfouri:2015,Gebarowski:1995,Perktas:2010,Ramos:2003,Unal:2001}%
. It is clear that the submanifolds $M_{1}\times \left\{ q\right\} $ and $%
\{p\}\times M_{2}$ are homothetic to $M_{1}$ and $M_{2}$ respectively for
each $p\in M_{1}$ and $q\in M_{2}$. We shall refer to these factor
submanifolds as $M_{1}$ and $M_{2}$. The lift $\bar{X}_{\left( p,q\right) }$
of a tangent vector $X_{p}\in T_{p}M_{1}$, $q\in M_{2}$ is the unique vector
in $T_{\left( p,q\right) }M$ such that%
\begin{equation*}
\pi _{1}^{\ast }\left( \bar{X}_{\left( p,q\right) }\right) =X_{p},\text{ \ \
\ \ }\pi _{2}^{\ast }\left( \bar{X}_{\left( p,q\right) }\right) =0
\end{equation*}%
Similarly, if $X_{i}\in \mathfrak{X}\left( M_{i}\right) $, then the lift of $%
X_{i}$ to $\mathfrak{X}\left( M_{1}\times M_{2}\right) $ is the unique
vector field in $\mathfrak{X}\left( M_{1}\times M_{2}\right) $ that is $\pi
_{i}-$related to $X_{i}$ and $\pi _{j}-$related to zero vector field in $%
\mathfrak{X}\left( M_{j}\right) ,$ $i\neq j$ i.e. a vector field $X_{i}$ on $%
M_{i}$ is identified with the horizontal or the vertical vector field on $%
M_{1}\times M_{2}$ that is $\pi _{i}-$related to $X_{i}$. Throughout this
article we use the same notation for a vector field and for its lift to the
product manifold. A function $\omega _{i}$ on $M_{i}$ will be identified
with $\omega _{i}\circ \pi _{i}$. Thus we have two different meanings for
the gradient of $\omega _{i}$, namely \textrm{grad}$\left( \omega _{i}\circ
\pi _{i}\right) \in $ $\mathfrak{X}\left( M_{1}\times M_{2}\right) $ and the
lift of the gradient $\nabla ^{i}\omega _{i}$ of $\omega _{i}$ to $\mathfrak{%
X}\left( M_{1}\times M_{2}\right) $. In fact we have%
\begin{equation*}
g\left( X_{i},\mathrm{grad}\left( \omega _{i}\circ \pi _{i}\right) \right)
=X_{i}\left( \omega _{i}\circ \pi _{i}\right) =X_{i}\left( \omega
_{i}\right) \circ \pi _{i}=\frac{1}{f_{j}^{2}}g\left( X_{i},\nabla
^{i}\omega _{i}\right)
\end{equation*}%
Therefore, \textrm{grad}$\left( \omega _{i}\circ \pi _{i}\right) =\frac{1}{%
f_{j}^{2}}\nabla ^{i}\omega _{i}$ (Note that we use the same notation for
the vector field $\nabla ^{i}\omega _{i}$ and for its lift to $\mathfrak{X}%
\left( M_{1}\times M_{2}\right) $).

Let $\left( M,g,D\right) $ is a pseudo-Riemannian doubly warped product
manifold of $\left( M_{i},g_{i},D_{i}\right) ,$ $i=1,2$ with dimensions $%
n_{i}$ where $n=n_{1}+n_{2}$. $R,R^{i}$ and {$\mathrm{Ric}$}, {$\mathrm{Ric}$%
}$^{i}$ denote the curvature tensor and Ricci curvature tensor on $M,M^{i}$
respectively. Moreover, $\nabla ^{i}f_{i},\bigtriangleup ^{i}f_{i}$ denote
gradient and Laplacian of $f_{i}$ on $M_{i}$ and $f_{i}^{\diamond
}=f_{i}\bigtriangleup ^{i}f_{i}+\left( n_{j}-1\right) g_{i}\left( \nabla
^{i}f_{i},\nabla ^{i}f_{i}\right) ,$ $i\neq j$. For the connection and
curvatures formulas of a pseudo-Riemannian doubly warped product manifold
see for example \cite{Agaoka:1998, Unal:PhD}.

A vector field $\zeta $ on a (pseudo-)Riemannian manifold $\left( N,h\right) 
$ with metric $h$ is called a conformal vector field with conformal factor $%
\rho $ if 
\begin{equation*}
\mathcal{L}_{\zeta }h=\rho h
\end{equation*}%
where $\mathcal{L}_{\zeta }$ is the Lie derivative on $N$ with respect to $%
\zeta $. If $\rho $ is constant or zero, $\zeta $ is called a homothetic or
Killing vector field on $N$ respectively. One can redefine conformal vector
fields using the following identity. Let $\zeta $ be a vector field on $M$,
then%
\begin{equation}
\left( \mathcal{L}_{\zeta }h\right) \left( X,Y\right) =h\left( D_{X}\zeta
,Y\right) +h\left( X,D_{Y}\zeta \right)
\end{equation}%
for any vector fields $X,Y\in \mathfrak{X}\left( N\right) $. A vector field $%
\zeta $ on a manifold $\left( N,h\right) $ is called concurrent if%
\begin{equation*}
D_{X}\zeta =X
\end{equation*}%
for any vector field $X\in \mathfrak{X}\left( N\right) $\cite{Chen2015}. Let 
$\zeta $ be a concurrent vector field, then%
\begin{equation*}
\left( \mathcal{L}_{\zeta }h\right) \left( X,Y\right) =2h\left( X,Y\right)
\end{equation*}%
and so $\zeta $ is homothetic with factor $\rho =2$. A zero vector field is
not concurrent. If both $\zeta $ and $\xi $ are concurrent vector fields,
then%
\begin{equation*}
D_{X}\left[ \zeta ,\xi \right] =0
\end{equation*}%
Also both $\zeta +\xi $ and $\lambda \zeta $ are not concurrent vector
fields. Finally, a Killing vector field is not concurrent. For example, a
vector field $\alpha \partial _{x}$ is a concurrent vector field on $\left( 
\mathbb{R},dx^{2}\right) $ if%
\begin{equation*}
D_{\partial _{x}}\left( \alpha \partial _{x}\right) =\partial _{x}
\end{equation*}%
i.e. $\alpha =x+a$. Thus concurrent vector fields on $\left( 
%TCIMACRO{\U{211d} }%
%BeginExpansion
\mathbb{R}
%EndExpansion
,dx^{2}\right) $ are of the form $\left( x+a\right) \partial _{x}$.

The following result represents a simple characterization of Killing vector
fields. If $\left( N,h\right) $ is a pseudo-Riemannian manifold with
Riemannian connection $D$. A vector field $\zeta \in \mathfrak{X}\left(
N\right) $ is a Killing vector field if and only if%
\begin{equation}
h\left( D_{X}\zeta ,X\right) =0
\end{equation}%
for any vector field $X\in \mathfrak{X}\left( N\right) $.

The following discussion represents a good tool to characterize Killing
vector fields on pseudo-Riemannian warped product manifolds. In \cite%
{Unal2012,Shenawy:2015}, the authors obtained many characterizations of
Killing vector fields on warped product manifolds and on standard static
spacetimes using these results. Let $M=M_{1}\times _{f}M_{2}$ be a
pseudo-Riemannian warped product manifold with warping function $f$. Let $%
\zeta =\zeta _{1}+\zeta _{2}\in \mathfrak{X}\left( M\right) $ be a vector
field on $M$. Then%
\begin{eqnarray}
g\left( D_{X}\zeta ,X\right) &=&g_{1}\left( D_{X_{1}}^{1}\zeta
_{1},X_{1}\right) +f^{2}g_{2}\left( D_{X_{2}}^{2}\zeta _{2},X_{2}\right)
+f\zeta _{1}\left( f\right) \left\Vert X_{2}\right\Vert _{2}^{2}  \notag \\
\left( \mathcal{L}_{\zeta }g\right) \left( X,Y\right) &=&\left( \mathcal{L}%
_{\zeta _{1}}^{1}g_{1}\right) \left( X_{1},Y_{1}\right) +f^{2}\left( 
\mathcal{L}_{\zeta _{2}}^{2}g_{2}\right) \left( X_{2},Y_{2}\right) +2f\zeta
_{1}\left( f\right) g_{2}\left( X_{2},Y_{2}\right)  \notag
\end{eqnarray}%
for any vector field $X=X_{1}+X_{2}\in \mathfrak{X}\left( M\right) $, where $%
\mathcal{L}_{\zeta _{i}}^{i}$ is the Lie derivative on $M_{i}$ with respect
to $\zeta _{i},$ for $i=1,2$.

A pseudo-Riemannian manifold $M$ is said to admit a Ricci curvature
collineation if there is a vector field $\zeta \in \mathfrak{X}\left(
M\right) $ such that%
\begin{equation*}
\mathcal{L}_{\zeta }\mathrm{Ric}=0
\end{equation*}%
where $\mathrm{Ric}$ is the Ricci curvature tensor\cite{Hall:2004}. Finally,
a spacetime $M$ is said to admit a matter collineation if there is a vector
field $\zeta \in \mathfrak{X}\left( M\right) $ such that%
\begin{equation*}
\mathcal{L}_{\zeta }\mathrm{T}=0
\end{equation*}%
where $\mathrm{T}$ is the energy-momentum tensor\cite{Carot:1994}. The
Einstein's field equation with cosmological constant $\lambda $ is given by%
\begin{equation*}
\mathrm{Ric}-\frac{r}{2}g=\kappa \mathrm{T}-\lambda g
\end{equation*}%
where $r$ is the scalar curvature. Suppose that $\zeta $ is Killing, then%
\begin{equation*}
\mathcal{L}_{\zeta }T=0
\end{equation*}%
i.e.$\zeta $ is a matter collineation whereas a matter collineation need not
be a Killing vector field. Also, a Killing vector field is a Ricci curvature
collineation. The converse is not generally true.

\section{Conformal vector fields on doubly warped products}

In this section we investigate the relation between conformal vector fields
on doubly warped product manifolds and those conformal vector fields on the
product factors. Throughout this section, let $M=_{f_{2}}M_{1}\times
_{f_{1}}M_{2}$ be a pseudo-Riemannian doubly warped product manifold with
the metric tensor $g=f_{2}^{2}g_{1}\oplus f_{1}^{2}g_{2}$ and $%
f_{i}:M_{i}\rightarrow \left( 0,\infty \right) $ is a smooth function where $%
i=1,2$ and $\left( M_{i},g_{i}\right) $ are pseudo-Riemannian manifolds. The
following result gives us an important identity to study such relation\cite%
{Unal:2001}.

\begin{proposition}
Suppose that $\zeta _{1},X_{1},Y_{1}\in \mathfrak{X}(M_{1})$ and $\zeta
_{2},X_{2},Y_{2}\in \mathfrak{X}(M_{2})$, then%
\begin{eqnarray}
\left( \mathcal{L}_{\zeta }g\right) (X,Y) &=&f_{2}^{2}\left( \mathcal{L}%
_{\zeta _{1}}^{1}g_{1}\right) (X_{1},Y_{1})+f_{1}^{2}\left( \mathcal{L}%
_{\zeta _{2}}^{2}g_{2}\right) (X_{2},Y_{2})  \notag \\
&&+2f_{1}\zeta _{1}\left( f_{1}\right) g_{2}(X_{2},Y_{2})+2f_{2}\zeta
_{2}\left( f_{2}\right) g_{1}(X_{1},Y_{1})  \label{ke1}
\end{eqnarray}
\end{proposition}

where $\zeta =\zeta _{1}+\zeta _{2},$ $X=X_{1}+X_{2}$ and $Y=Y_{1}+Y_{2}$
are elements in $\mathfrak{X}(M)$.

In \cite{Unal:2001}, the author considered a characterization of conformal
vector fields on doubly warped product manifolds. In fact, it is just a
characterization of homothetic vector fields. The following theorem
represents a new characterization of conformal vector fields on doubly
warped product manifolds but the assumption here is less restrictive.

\begin{theorem}
\label{Thm1}A vector field $\zeta =\zeta _{1}+\zeta _{2}$ on a
pseudo-Riemannian doubly warped product $M=_{f_{2}}M_{1}\times _{f_{1}}M_{2}$
is a conformal vector field with conformal factor $\rho $ if and only if

\begin{enumerate}
\item $\zeta _{i}$ is a conformal vector field on $M_{i}$ with conformal
factor $\rho _{i},i=1,2$, and

\item $\rho _{1}+2\zeta _{2}\left( \ln f_{2}\right) =\rho _{2}+2\zeta
_{1}\left( \ln f_{1}\right) $
\end{enumerate}

Moreover, the conformal factor of $\zeta $ is $\rho =\rho _{i}+2\zeta
_{j}\left( \ln f_{j}\right) ,i\neq j$.
\end{theorem}

Before proceeding further, one may notice that a doubly warped product
metric $g$ on $M$ can be expressed as a conformal metric to a product metric
on $M_{1}\times M_{2}$ as follows:%
\begin{equation*}
g=f_{1}^{2}f_{2}^{2}\left( \frac{1}{f_{1}^{2}}g_{1}+\frac{1}{f_{2}^{2}}%
g_{2}\right) =f_{1}^{2}f_{2}^{2}\left( \bar{g}_{1}+\bar{g}_{2}\right)
=f_{1}^{2}f_{2}^{2}\bar{g}
\end{equation*}%
Let us consider the effect of replacing the metric $g$ on $M$ by $\bar{g}=%
\bar{g}_{1}+\bar{g}_{2}$. A similar discussion on $4-$dimensional spacetimes
is considered in \cite[Chapter 11]{Hall:2004}. Suppose that $\zeta =\zeta
_{1}+\zeta _{2}$ is a conformal vector field on $\left( M,\bar{g}\right) $
with factor $\bar{\rho}$, then%
\begin{equation*}
\mathcal{L}_{\zeta }g=\left[ 2\zeta _{2}\left( \ln f_{2}\right) +2\zeta
_{1}\left( \ln f_{1}\right) +\bar{\rho}\right] g
\end{equation*}%
Therefore, $\zeta $ is a conformal vector field on $\left( M,g\right) $ with
factor $\rho =\bar{\rho}\mathcal{+}2\zeta _{2}\left( \ln f_{2}\right)
+2\zeta _{1}\left( \ln f_{1}\right) $. A similar conclusion applies to $%
\left( M_{i},g_{i}\right) $ and $\left( M_{i},\bar{g}_{i}\right) $ where $%
\bar{\rho}_{i}=\rho -\zeta _{i}\left( \ln f_{i}\right) $. Thus, by using
results in \cite[Theorem 1]{Apostolopoulos:2005}, one can easily get the
following:

\begin{theorem}
Let $M=_{f_{2}}M_{1}\times _{f_{1}}M_{2}$ be a pseudo-Riemannian doubly
warped product equipped with the metric tensor $%
g=f_{2}^{2}g_{1}+f_{1}^{2}g_{2}$ and let $\bar{g}=\bar{g}_{1}+\bar{g}_{2}$
where $\bar{g}_{i}=\frac{1}{f_{i}^{2}}g_{i}$. Then,

\begin{enumerate}
\item a Killing vector field $\zeta =\zeta _{i}$ on $\left( M_{i},\bar{g}%
_{i}\right) $, for each $i=1,2$, is a Killing vector field on $\left( M,\bar{%
g}\right) $,

\item $\left( M,\bar{g}\right) $ admits a homothetic vector field if and
only if $\left( M_{i},\bar{g}_{i}\right) $ admits a homothetic vector field
for each $i=1,2$,

\item each conformal vector field on $\left( M,\bar{g}\right) $ is a
conformal vector field on $\left( M,g\right) $.
\end{enumerate}
\end{theorem}

The above theorem together with Theorem \ref{Thm1} imply the following
results.

\begin{theorem}
Let $\zeta =\zeta _{1}+\zeta _{2}$ be a vector field on be a
pseudo-Riemannian doubly warped product $M=_{f_{2}}M_{1}\times _{f_{1}}M_{2}$
equipped with the metric tensor $g=f_{2}^{2}g_{1}+f_{1}^{2}g_{2}$. Assume
that $\zeta _{i}$ is a Killing vector field on $\left( M_{i},g_{i}\right) $
for each $i=1,2$ and $\zeta _{1}\left( \ln f_{1}\right) =\zeta _{2}\left(
\ln f_{2}\right) $. Then $\zeta $ is a conformal vector field on $M$.
\end{theorem}

\begin{theorem}
Let $\zeta _{i}\in \mathfrak{X}\left( M_{i}\right) $ be homothetic vector
fields on $\left( M_{i},g_{i}\right) $ with factors $a_{i}$ for each $i=1,2$%
. Assume that $\zeta _{1}\left( f_{1}\right) =\zeta _{2}\left( f_{2}\right)
=0$. Then $\zeta =a_{2}\zeta _{1}+a_{1}\zeta _{2}$ is a homothetic vector
field on $\left( M,g\right) $ with factor $a_{1}a_{2}$.
\end{theorem}

\begin{corollary}
The dimension of the conformal group $C\left( M,g\right) $ on a
pseudo-Riemannian doubly warped product $M=_{f_{2}}M_{1}\times _{f_{1}}M_{2}$
is at least%
\begin{equation*}
\mathrm{dim}K_{1}\left( M_{1},\bar{g}_{1}\right) +\mathrm{dim}K_{1}\left(
M_{2},\bar{g}_{2}\right) 
\end{equation*}%
where $K_{i}\left( M_{i},\bar{g}_{i}\right) $ is the isometry group of $%
\left( M_{i},\bar{g}_{i}\right) $.
\end{corollary}

Again Theorem \ref{Thm1} together with Lemma 2.1 in \cite{Kuhnel:1997} yield
the following result.

\begin{theorem}
Let $\zeta =\zeta _{1}+\zeta _{2}$ be a vector field on a pseudo-Riemannian
doubly warped product $M=_{f_{2}}M_{1}\times _{f_{1}}M_{2}$ such that

\begin{enumerate}
\item $\zeta _{i}$ is a conformal vector field on $M_{i}$ with conformal
factor $\rho _{i},i=1,2$, and

\item $\rho _{1}+2\zeta _{2}\left( \ln f_{2}\right) =\rho _{2}+2\zeta
_{1}\left( \ln f_{1}\right) $
\end{enumerate}

Then $\zeta $ preserves the Ricci curvature if and only if $H^{u}=0$ where $%
u=\rho _{2}+2\zeta _{1}\left( \ln f_{1}\right) $. Moreover, $\zeta $
preserves the conformal class of the Ricci tensor(i.e. $\mathcal{L}_{\zeta }%
\mathrm{Ric}=\lambda g$ for some function $\lambda $) if and only if $\nabla
\left( div\left( \zeta \right) \right) $ is a conformal vector field.
\end{theorem}

\begin{theorem}
Let $\zeta =\zeta _{1}+\zeta _{2}$ be a vector field on a pseudo-Riemannian
doubly warped product $M=_{f_{2}}M_{1}\times _{f_{1}}M_{2}$. $\zeta $ has
constant length along the integral curve $\alpha $ of the vector field $%
X=X_{1}+X_{2}\in \mathfrak{X}\left( M\right) $ if one of the following
conditions holds
\end{theorem}

\begin{enumerate}
\item $X_{i}\left( f_{i}\right) =0$ and $\zeta _{i}$ is parallel along $\pi
_{i}\circ \alpha $ for each $i=1,2$.

\item $X_{i}\left( f_{i}\right) =0$ and $\zeta _{i}$ has a constant length
along $\pi _{i}\circ \alpha $ for each $i=1,2$.
\end{enumerate}

\begin{proof}
Let $\zeta =\zeta _{1}+\zeta _{2}\in \mathfrak{X}\left( M\right) $ be a
vector field on $M$. Then%
\begin{eqnarray*}
g(D_{X}\zeta ,\zeta ) &=&g(D_{X_{1}+X_{2}}\zeta _{1}+\zeta _{2},\zeta ) \\
&=&g(D_{X_{1}}\zeta _{1}+D_{X_{1}}\zeta _{2}+D_{X_{2}}\zeta
_{1}+D_{X_{2}}\zeta _{2},\zeta ) \\
&=&f_{2}^{2}g_{1}\left( D_{X_{1}}^{1}\zeta _{1},\zeta _{1}\right)
+f_{1}^{2}g_{2}\left( D_{X_{2}}^{2}\zeta _{2},\zeta _{2}\right) \\
&&+f_{1}X_{1}\left( f_{1}\right) \left\Vert \zeta _{2}\right\Vert
_{2}^{2}+f_{2}X_{2}(f_{2})\left\Vert \zeta _{1}\right\Vert _{1}^{2}
\end{eqnarray*}%
In both cases $X_{i}\left( f_{i}\right) =0$, and hence 
\begin{equation*}
g(D_{X}\zeta ,\zeta )=f_{2}^{2}g_{1}\left( D_{X_{1}}^{1}\zeta _{1},\zeta
_{1}\right) +f_{1}^{2}g_{2}\left( D_{X_{2}}^{2}\zeta _{2},\zeta _{2}\right)
\end{equation*}%
The first condition implies that $D_{X_{i}}^{i}\zeta _{i}=0$ and so $%
g_{i}\left( D_{X_{i}}^{i}\zeta _{i},\zeta _{i}\right) =0$. The second
condition implies that $g_{i}\left( \zeta _{i},\zeta _{i}\right) $ is
constant and so%
\begin{equation*}
=2g_{i}\left( D_{X_{i}}^{i}\zeta _{i},\zeta _{i}\right)
\end{equation*}

and therefore%
\begin{equation*}
g(D_{X}\zeta ,\zeta )=0
\end{equation*}%
i.e. $\zeta $ has a constant length along the integral curve $\alpha $ of
the vector field $X$
\end{proof}

\begin{theorem}
\label{th4}Let $\zeta =\zeta _{1}+\zeta _{2}\in \mathfrak{X}(M)$ be a
conformal vector field on a pseudo-Riemannian doubly warped product $%
M=_{f_{2}}M_{1}\times _{f_{1}}M_{2}$ along a curve $\alpha $ with unit
tangent vector $T=V_{1}+V_{2}$. Then%
\begin{equation*}
div\left( \zeta \right) =n\left[ f_{2}^{2}g_{1}\left( D_{V_{1}}^{1}\zeta
_{1},V_{1}\right) +f_{1}^{2}g_{2}\left( D_{V_{2}}^{2}\zeta _{2},V_{2}\right)
+f_{2}\zeta _{2}\left( f_{2}\right) \left\Vert V_{1}\right\Vert
_{1}^{2}+f_{1}\zeta _{1}\left( f_{1}\right) \left\Vert V_{2}\right\Vert
_{2}^{2}\right]
\end{equation*}
\end{theorem}

\begin{proof}
Let $\zeta $ be a conformal vector field with conformal factor $\rho $. Then%
\begin{equation*}
\left( \mathcal{L}_{\zeta }g\right) (X,Y)=\rho g\left( X,Y\right)
\end{equation*}%
Let $X=Y=T,$ then%
\begin{equation*}
2g(D_{T}\zeta ,T)=\rho g\left( T,T\right)
\end{equation*}%
then the conformal factor $\rho $ is given by%
\begin{equation*}
\rho =2g\left( D_{T}\zeta ,T\right)
\end{equation*}%
Suppose that $\zeta =\zeta _{1}+\zeta _{2}$ and $T=V_{1}+V_{2}$, then%
\begin{eqnarray*}
\rho &=&2g\left( D_{V_{1}}\zeta _{1}+D_{V_{1}}\zeta _{2}+D_{V_{2}}\zeta
_{1}+D_{V_{2}}\zeta _{2},T\right) \\
&=&2f_{2}^{2}g_{1}\left( D_{V_{1}}^{1}\zeta _{1},V_{1}\right)
+2f_{1}^{2}g_{2}\left( D_{V_{2}}^{2}\zeta _{2},V_{2}\right) \\
&&+2f_{2}\zeta _{2}(f_{2})\left\Vert V_{1}\right\Vert _{1}^{2}+2f_{1}\zeta
_{1}\left( f_{1}\right) \left\Vert V_{2}\right\Vert _{2}^{2}
\end{eqnarray*}%
But the conformal factor is given by 
\begin{equation*}
2div\left( \zeta \right) =\rho n
\end{equation*}%
which completes the proof.
\end{proof}

\subsection{Conformal vector fields on doubly warped spacetimes}

A doubly warped spacetime is a doubly warped product manifold $%
M=_{f_{2}}M_{1}\times _{f_{1}}M_{2}$ where one of the factors, say $M_{1}$,
has a Lorentz signature and the second is Riemannian. Ramos et al.
considered an invariant characterization of $4-$dimensional doubly warped
spacetimes\cite{Ramos:2003}. Among many other results, they obtained
necessary and sufficient conditions for a (locally) double warped spacetime
to be conformally related to a $1+3$ or $2+2$ decomposable spacetime. Then
they studied the conformal algebra of $2+2$ decomposable spacetimes in
section IV and $1+3$ decomposable spacetimes in section V. For a detailed
discussion of conformally related $1+3$ and $2+2$ reducible spacetimes see 
\cite{Carot:2002, Carot:2008} and for an extensive self contained study of
conformal symmetry of $4-$dimensional spacetimes, the reader is referred to 
\cite{Hall:2004}.

We restrict our study of conformal vector fields on doubly warped spacetimes
to \textrm{dim}$\left( M_{1}\right) \leq 2$ since this case generalizes some
well-known exact solutions of the Einstein field equations and the beginning
of this section represents such study irrespective of the dimension of the
factors. For this case, either \textrm{dim}$\left( M_{1}\right) =1$ or 
\textrm{dim}$\left( M_{1}\right) =2$. In the following we deal with both
subcases separately. Let us first consider doubly warped spacetimes with a $%
1-$dimensional base.

Let $(M,g)$ be a Riemannian manifold and $I$ be an open connected interval
equipped with the metric $-\mathrm{dt}^{2}$. A doubly warped spacetime $\bar{%
M}=_{f}I\times _{\sigma }M$ is the product $I\times M$ furnished with the
metric%
\begin{equation*}
\bar{g}=-f^{2}\mathrm{dt}^{2}\oplus \sigma ^{2}g
\end{equation*}%
where $f:M\rightarrow \left( 0,\infty \right) $ and $\sigma :I\rightarrow
\left( 0,\infty \right) $ are smooth functions. $\bar{M}$ is a generalized
Robertson-Walker spacetime if $f$ is constant and a standard static
spacetime if $\sigma $ is constant.

An investigation of $4-$dimensional spacetimes that are conformally related
to $1+3$ reducible spacetimes was carried out with many examples in the
aforementioned references \cite{Carot:2008,Ramos:2003}. A classification of
these spacetimes according to their conformal algebra is considered in the
first reference whereas a special attention is paid to gradient conformal
vector fields in the second reference.

\begin{theorem}
A time-like vector field $\bar{\zeta}=h\partial _{t}$ is a conformal vector
field on a doubly warped spacetime $\bar{M}=_{f}I\times _{\sigma }M$ if and
only if $h=a\sigma $ where $a$ is a non-negative constant. Moreover, the
conformal factor is $\rho =2\dot{h}$.
\end{theorem}

\begin{proof}
If $h=0$, then $a=0$ and the result is obvious. Now, we assume that $h\neq 0$%
. Using equation (\ref{ke1}), we get that%
\begin{eqnarray*}
\left( \mathcal{\bar{L}}_{\bar{\zeta}}\bar{g}\right) (\bar{X},\bar{Y})
&=&-2xyf^{2}\left[ \dot{h}+\zeta \left( \ln f\right) \right] +\sigma
^{2}\left( \mathcal{L}_{\zeta }g\right) (X,Y)+2h\sigma \dot{\sigma}g\left(
X,Y\right) \\
&=&-2xyf^{2}\dot{h}+2h\sigma \dot{\sigma}g\left( X,Y\right) \\
&=&2\dot{h}f^{2}g_{I}\left( x\partial _{t},y\partial _{t}\right) +\frac{2h%
\dot{\sigma}}{\sigma }\sigma ^{2}g\left( X,Y\right)
\end{eqnarray*}%
Suppose that $h=a\sigma $, then%
\begin{equation*}
\left( \mathcal{\bar{L}}_{\bar{\zeta}}\bar{g}\right) (\bar{X},\bar{Y})=2\dot{%
h}\bar{g}(\bar{X},\bar{Y})
\end{equation*}%
i.e.$\bar{\zeta}=h\partial _{t}$ is a conformal vector field with conformal
factor $\rho =2\dot{h}$. Conversely, suppose that $\overline{\zeta }%
=h\partial _{t}\in \mathfrak{X}(\bar{M})$ is a conformal vector field with
factor $\rho ,$ then%
\begin{equation*}
\left( \mathcal{\bar{L}}_{\bar{\zeta}}\bar{g}\right) (\bar{X},\bar{Y})=\rho 
\bar{g}(\bar{X},\bar{Y})
\end{equation*}%
for any vector fields $\overline{X},\overline{Y}\in \mathfrak{X}(\overline{M}%
)$. Now, by equation \ref{ke1}, we get that%
\begin{equation*}
\rho \bar{g}(\bar{X},\bar{Y})=2\dot{h}f^{2}g_{I}\left( x\partial
_{t},y\partial _{t}\right) +\frac{2h\dot{\sigma}}{\sigma }\sigma ^{2}g\left(
X,Y\right)
\end{equation*}%
Let $X=Y=0$, we get that%
\begin{equation*}
\rho \bar{g}(x\partial _{t},y\partial _{t})=2\dot{h}f^{2}g_{I}\left(
x\partial _{t},y\partial _{t}\right)
\end{equation*}%
i.e. $\rho =2\dot{h}$. Now, let $x=y=0$, we get that $\rho \sigma =2h\dot{%
\sigma}$.

These two differential equations imply that $h=a\sigma $ for some positive
constant $a$.
\end{proof}

\begin{theorem}
\label{th1}A vector field $\bar{\zeta}=h\partial _{t}+\zeta $ is a Killing
vector field on a doubly warped spacetime $\bar{M}=_{f}I\times _{\sigma }M$
if and only if\ one of the following conditions hold

\begin{enumerate}
\item $\bar{\zeta}$ is time-like and $\dot{h}=\dot{\sigma}=0$

\item $\bar{\zeta}$ is space-like where $\zeta $ is a Killing vector field
on $M$ and $\zeta \left( f\right) =0$

\item $\dot{h}=-\zeta \left( \ln f\right) $ and $\zeta $ is a conformal
vector field on $M$ with conformal factor $\rho _{2}=-\dfrac{2h\dot{\sigma}}{%
\sigma }$
\end{enumerate}
\end{theorem}

\begin{proof}
The first assertion is a special case of the above theorem. For the second
assertion, let $h=0$ in equation (\ref{ke1}). Thus%
\begin{equation*}
\left( \mathcal{\bar{L}}_{\bar{\zeta}}\bar{g}\right) (\bar{X},\bar{Y}%
)=-2xyf\zeta \left( f\right) +\sigma ^{2}\left( \mathcal{L}_{\zeta }g\right)
(X,Y)
\end{equation*}%
Suppose that $\zeta $ is a Killing vector field on $M$ and $\zeta \left(
f\right) =0,$ then%
\begin{equation*}
\left( \mathcal{\bar{L}}_{\bar{\zeta}}\bar{g}\right) (\bar{X},\bar{Y})=0
\end{equation*}

The converse is direct. Finally, let $\bar{\zeta}$ be a Killing vector field
on $\bar{M}=_{f}I\times _{\sigma }M$, then%
\begin{eqnarray*}
0 &=&\left( \mathcal{\bar{L}}_{\bar{\zeta}}\bar{g}\right) (\bar{X},\bar{Y})
\\
0 &=&-2xyf^{2}\left[ \dot{h}+\zeta \left( \ln f\right) \right] +\sigma
^{2}\left( \mathcal{L}_{\zeta }g\right) (X,Y)+2h\sigma \dot{\sigma}g\left(
X,Y\right)
\end{eqnarray*}%
Let $X=Y=0,$then%
\begin{equation*}
-2xyf^{2}\left[ \dot{h}+\zeta \left( \ln f\right) \right] =0
\end{equation*}%
and so $\dot{h}=-\zeta \left( \ln f\right) $. Thus%
\begin{equation*}
\left( \mathcal{L}_{h\partial _{t}}^{I}g_{I}\right) (x\partial
_{t},y\partial _{t})=-2\zeta \left( \ln f\right) g_{I}(x\partial
_{t},y\partial _{t})
\end{equation*}%
i.e. $h\partial _{t}$ is a conformal vector field on $I$ with conformal
factor $-2\zeta \left( \ln f\right) $. Now let $x=y=0$, then%
\begin{equation*}
\left( \mathcal{L}_{\zeta }g\right) (X,Y)=-\frac{2h\dot{\sigma}}{\sigma }%
g\left( X,Y\right)
\end{equation*}%
i.e. $\zeta $ is a conformal vector field on $M$ with conformal factor $\rho
=-\frac{2h\dot{\sigma}}{\sigma }$.

Conversely, suppose that $\dot{h}=-\zeta \left( \ln f\right) $ and $\zeta $
is a conformal vector field on $M.$with conformal factor $\rho =-\frac{2h%
\dot{\sigma}}{\sigma },$ then%
\begin{equation*}
\left( \mathcal{L}_{\zeta }g\right) (X,Y)=-\frac{2h\dot{\sigma}}{\sigma }%
g\left( X,Y\right)
\end{equation*}%
Thus, 
\begin{eqnarray*}
\left( \mathcal{\bar{L}}_{\bar{\zeta}}\bar{g}\right) (\bar{X},\bar{Y})
&=&-2xyf^{2}\left[ \dot{h}+\zeta \left( \ln f\right) \right] +\sigma
^{2}\left( \mathcal{L}_{\zeta }g\right) (X,Y)+2h\sigma \dot{\sigma}g\left(
X,Y\right) \\
&=&0
\end{eqnarray*}%
i.e.$\bar{\zeta}=h\partial _{t}+\zeta $ is a Killing vector field on $\bar{M}%
=_{f}I\times _{\sigma }M$.
\end{proof}

\begin{corollary}
Let $\bar{\zeta}=h\partial _{t}+\zeta $ be vector field on a doubly warped
spacetime $\bar{M}=_{f}I\times _{\sigma }M$ obeying Einstein's field
equation. Then

\begin{enumerate}
\item $\bar{M}$ admits a time-like matter collineation $\bar{\zeta}%
=h\partial _{t}$ if $\dot{h}=\dot{\sigma}=0$

\item $\bar{M}$ admits a space-like matter collineation $\bar{\zeta}=\zeta $
if $\zeta $ is a Killing vector field on $M$ and $\zeta \left( f\right) =0$

\item $\bar{M}$ admits a matter collineation $\bar{\zeta}=h\partial
_{t}+\zeta $ if $\dot{h}=-\zeta \left( \ln f\right) $ and $\zeta $ is a
conformal vector field on $M$ with conformal factor $\rho _{2}=-\dfrac{2h%
\dot{\sigma}}{\sigma }$
\end{enumerate}
\end{corollary}

The following result is a direct consequence of Theorem \ref{th4}.

\begin{corollary}
Let $\bar{\zeta}=h\partial _{t}+\zeta \in \mathfrak{X}(\bar{M})$ be a
conformal vector field on a doubly warped spacetime $\bar{M}=_{f}I\times
_{\sigma }M$ along a curve $\alpha $ with unit tangent vector $\bar{V}%
=v\partial _{t}+V$. Then the conformal factor $\rho $ of $\bar{\zeta}$ is
given by%
\begin{equation*}
\rho =2\dot{h}+2\sigma ^{2}g\left( \left[ \zeta ,V\right] ,V\right) +2\left(
h\sigma \dot{\sigma}-\dot{h}\sigma ^{2}\right) g\left( V,V\right)
\end{equation*}
\end{corollary}

In the sequel, we present doubly warped spacetimes with a $2-$dimensional
base. In this subcase, $2+n$ doubly warped spacetimes are considered to be
doubly warped product manifolds with a $2-$dimensional pseudo-Riemannian
base and an $n-$dimensional Riemannian fibre. A $2+n$ doubly warped
spacetime is clearly conformal to a product manifold. In \cite%
{Carot:2002,Ramos:2003} and references therein, the Lie conformal algebra of
conformally related $2+2$ reducible spacetimes is extensively studied. Many
interesting results and examples are given there. For example, in \cite%
{Carot:2002}, Carot and Tupper considered an invariant characterization that
imposes conditions on the conformal factor and on two null vectors.
Moreover, N Van den Bergh considered non-conformally flat perfect fluids
spacetimes which are conformally $2+2-$decomposable spacetimes with factor
spaces of constant curvature\cite{Bergh:2012}.

It is well-known that each $2-$dimensional manifold is conformally flat.
Thus we may simply take the base manifold as $\left( 
%TCIMACRO{\U{211d} }%
%BeginExpansion
\mathbb{R}
%EndExpansion
^{2},\mathrm{ds}^{2}\right) $ where $\mathrm{ds}^{2}=-dt^{2}+dx^{2}$. Let $%
\bar{M}=_{f}%
%TCIMACRO{\U{211d} }%
%BeginExpansion
\mathbb{R}
%EndExpansion
^{2}\times _{\sigma }M$ be $\left( 2+n\right) -$dimensional doubly warped
product spacetime furnished with the metric $\bar{g}=f^{2}ds^{2}+\sigma
^{2}g $.

\begin{proposition}
Suppose that $h\left( t\right) \partial _{t}+u\left( x\right) \partial
_{x},h_{i}\left( t\right) \partial _{t}+u_{i}\left( x\right) \partial
_{x}\in \mathfrak{X}(%
%TCIMACRO{\U{211d} }%
%BeginExpansion
\mathbb{R}
%EndExpansion
^{2}),i=1,2$ and $\zeta ,X_{1},X_{2}\in \mathfrak{X}\left( M\right) $, then%
\begin{eqnarray}
\left( \mathcal{\bar{L}}_{\bar{\zeta}}\bar{g}\right) (\bar{X},\bar{Y})
&=&2f^{2}\left( -h_{1}h_{2}\dot{h}+u_{1}u_{2}u^{\prime }\right) +\sigma
^{2}\left( \mathcal{L}_{\zeta }g\right) (X,Y)  \notag \\
&&+2\sigma \left( h\sigma _{t}+u\sigma _{x}\right) g(X,Y)+2f\zeta \left(
f\right) \left( -h_{1}h_{2}+u_{1}u_{2}\right)
\end{eqnarray}%
where $\bar{\zeta}=h\partial _{t}+u\partial _{x}+\zeta $ and $\bar{X}%
_{i}=h_{i}\partial _{t}+u_{i}\partial _{x}+X_{i}$ are elements in $\mathfrak{%
X}\left( \bar{M}\right) $.
\end{proposition}

\begin{corollary}
A vector field $\bar{\zeta}=h\left( t\right) \partial _{t}+u\left( x\right)
\partial _{x}+\zeta \in \mathfrak{X}\left( \bar{M}\right) $ is a Killing
vector field if one of the following conditions hold

\begin{enumerate}
\item $\zeta =0$, $h=a$, $u=b$ and $a\sigma _{t}+b\sigma _{x}=0$.

\item $\zeta $ is Killing on $M$, $h=u=0$ and $\zeta \left( f\right) =0$
\end{enumerate}
\end{corollary}

\subsection{Concurrent vector fields on doubly warped spacetimes}

In this subsection we study concurrent vector fields on doubly warped
spacetimes with a $1-$dimensional base. One can extend most of the results
to doubly warped spacetimes with a $2-$dimensional base. Throughout this
subsection, let $\bar{M}=_{f}I\times _{\sigma }M$ be a doubly warped
spacetime equipped with the metric tensor $\bar{g}=-f^{2}dt^{2}\oplus \sigma
^{2}g$.

\begin{theorem}
A vector field $\bar{\zeta}=h\partial _{t}+\zeta $ on a doubly warped
spacetime $\bar{M}=_{f}I\times _{\sigma }M$ is a concurrent vector field if

\begin{enumerate}
\item $\zeta $ and $h\partial _{t}$ are concurrent vector fields on $M$ and $%
I$ respectively, and

\item both $f$ and $\sigma $ are constant.
\end{enumerate}
\end{theorem}

\begin{proof}
Suppose that $\bar{X}=x\partial _{t}+X\in \mathfrak{X}(\bar{M})$ is any
vector field on $\bar{M}$, then%
\begin{eqnarray*}
\bar{D}_{\bar{X}}\bar{\zeta} &=&\bar{D}_{x\partial _{t}}h\partial _{t}+\bar{D%
}_{x\partial _{t}}\zeta +\bar{D}_{X}h\partial _{t}+\bar{D}_{X}\zeta \\
&=&x\dot{h}\partial _{t}+\frac{xhf}{\sigma ^{2}}\nabla f+x(\ln \sigma \dot{)}%
\zeta +\zeta (\ln f)x\partial _{t} \\
&&+h(\ln \sigma \dot{)}X+X(\ln f)h\partial _{t}+D_{X}\zeta -\frac{\sigma 
\dot{\sigma}}{f^{2}}g\left( X,\zeta \right) \partial _{t} \\
&=&\left[ x\dot{h}+\zeta (\ln f)x+X(\ln f)h-\frac{\sigma \dot{\sigma}}{f^{2}}%
g\left( X,\zeta \right) \right] \partial _{t} \\
&&+\frac{xhf}{\sigma ^{2}}\nabla f+x(\ln \sigma \dot{)}\zeta +h(\ln \sigma 
\dot{)}X+D_{X}\zeta
\end{eqnarray*}%
Now suppose that $h\partial _{t}$ and $\zeta $ are concurrent vector field,
then $\dot{h}=1$ and $D_{X}\zeta =X$. If both $\sigma $ and $f$ are
constant, then%
\begin{eqnarray*}
\bar{D}_{\bar{X}}\bar{\zeta} &=&x\dot{h}\partial _{t}+D_{X}\zeta \\
&=&\bar{X}
\end{eqnarray*}%
i.e. $\bar{\zeta}$ is concurrent.
\end{proof}

It is well-known that a homothetic vector field is a matter collineation.
Thus the above theorem yields the following result.

\begin{corollary}
A vector field $\bar{\zeta}=h\partial _{t}+\zeta $ on a doubly warped
spacetime $\bar{M}=_{f}I\times _{\sigma }M$ is a matter collineation if

\begin{enumerate}
\item $\zeta $ and $h\partial _{t}$ are concurrent vector fields on $M$ and $%
I$ respectively, and

\item both $f$ and $\sigma $ are constant.
\end{enumerate}
\end{corollary}

\begin{theorem}
Let $\bar{\zeta}=h\partial _{t}+\zeta $ be a concurrent vector field on a
doubly warped spacetime $\bar{M}=_{f}I\times _{\sigma }M$ equipped with the
metric tensor $\bar{g}=-f^{2}dt^{2}\oplus \sigma ^{2}g$. Then $\zeta $ is a
concurrent vector field on $M$ if one of the following conditions holds.

\begin{enumerate}
\item $h=0$, or

\item $\sigma $ is a constant i.e. $\bar{M}$ is a standard static spacetime.
\end{enumerate}

Moreover, condition (1) implies condition (2) and the converse is true if $f$
is not constant.
\end{theorem}

\begin{proof}
From the above proof we have%
\begin{equation}
x=x\dot{h}+\zeta (\ln f)x+X(\ln f)h-\frac{\sigma \dot{\sigma}}{f^{2}}g\left(
X,\zeta \right)
\end{equation}%
and%
\begin{equation}
X=\frac{xhf}{\sigma ^{2}}\nabla f+x(\ln \sigma \dot{)}\zeta +h(\ln \sigma 
\dot{)}X+D_{X}\zeta
\end{equation}%
for any $x$ and $X$. Let $x=0$, then%
\begin{equation*}
0=X(\ln f)h-\frac{\sigma \dot{\sigma}}{f^{2}}g\left( X,\zeta \right)
\end{equation*}%
and%
\begin{eqnarray*}
X &=&h(\ln \sigma \dot{)}X+D_{X}\zeta \\
D_{X}\zeta &=&\left[ 1-\frac{h\dot{\sigma}}{\sigma }\right] X
\end{eqnarray*}%
Thus $\zeta $ is concurrent if $h\dot{\sigma}=0$.

If $h=0$, then%
\begin{equation*}
\frac{\sigma \dot{\sigma}}{f^{2}}g\left( X,\zeta \right) =0
\end{equation*}%
If $g\left( X,\zeta \right) =0$ then $\bar{\zeta}=0$ which is a
contradiction and so $\dot{\sigma}=0$ i.e. $\bar{M}$ is a standard static
spacetime.

If $\dot{\sigma}=0$, then%
\begin{equation*}
X(f)h=0
\end{equation*}%
which implies that $h=0$ for a non-constant function $f$.
\end{proof}

\begin{theorem}
Let $\bar{\zeta}=h\partial _{t}+\zeta $ be a concurrent vector field on $%
\bar{M}$ where $h\neq 0$. Then $\zeta $ and $h\partial _{t}$ are concurrent
vector fields on $M$ and $I$ respectively if $\dot{\sigma}=0$. In this case $%
f$ is also constant.
\end{theorem}

\begin{example}
The following table summarizes all three cases of concurrent vector fields
on the $2-$dimensional doubly warped spacetime of the form $\bar{M}%
=_{f}I\times _{\sigma }%
%TCIMACRO{\U{211d} }%
%BeginExpansion
\mathbb{R}
%EndExpansion
$ equipped with the metric $\bar{g}=-f^{2}dt^{2}\oplus \sigma ^{2}dx^{2}$.
For more details see Appendix A.
\end{example}

\begin{center}
\begin{tabular}{|c|c|c|c|c|}
\hline
\multicolumn{2}{|c|}{Case} & $\zeta $ & $\sigma $ & $f$ \\ \hline
$h=0$ & $\dot{\sigma}=0$ & $\zeta =\left( x+a\right) \partial _{x}$ & 
constant & $r\left( x+a\right) $ \\ \hline
$k=0$ & $f^{\prime }=0$ & $\zeta =\left( t+a\right) \partial _{t}$ & $%
r\left( t+a\right) $ & constant \\ \hline
$f^{\prime }=0$ & $\dot{\sigma}=0$ & $\zeta =\left( t+a\right) \partial
_{t}+\left( x+a\right) \partial _{x}$ & constant & constant \\ \hline
\end{tabular}
\end{center}

\section{Ricci Soliton on doubly warped spacetimes}

A smooth vector field $\zeta $ on a Riemannian manifold $(M,g)$ is said to
define a Ricci soliton if%
\begin{equation*}
\frac{1}{2}\left( \mathcal{L}_{\zeta }g\right) \left( X,Y\right) +\mathrm{Ric%
}\left( X,Y\right) =\lambda g\left( X,Y\right)
\end{equation*}

where $\mathcal{L}_{\zeta }$ denotes the Lie-derivative of the metric tensor 
$g$, $\mathrm{Ric}$ is the Ricci curvature and $\lambda $ is a constant\cite%
{Barros:2012,Barros:2013,Fernandez:2011,Munteanu:2013,Peterson:2009}.

\begin{theorem}
\label{th2}Let $\left( \bar{M},\bar{g},\bar{\zeta},\lambda \right) $ be a
Ricci soliton where $\bar{M}=_{f}I\times _{\sigma }M$ is a doubly warped
spacetime\ and $\bar{\zeta}=h\partial _{t}+\zeta \in \mathfrak{X}\left( \bar{%
M}\right) $. Then%
\begin{equation*}
\dot{h}=\frac{1}{f^{2}}\left( \lambda f^{2}-f\zeta \left( f\right) +\frac{n}{%
\sigma }\ddot{\sigma}+\frac{f^{\diamond }}{\sigma ^{2}}\right)
\end{equation*}%
and%
\begin{equation*}
\frac{1}{2}\sigma ^{2}\left( \mathcal{L}_{\zeta }g\right) (X,Y)+\mathrm{Ric}%
\left( X,Y\right) -\frac{1}{f}H^{f}\left( X,Y\right) =\left( \lambda \sigma
^{2}-h\sigma \dot{\sigma}+\frac{\sigma ^{\diamond }}{f^{2}}\right) g\left(
X,Y\right)
\end{equation*}
\end{theorem}

\begin{proof}
Let $\bar{M}=_{f}I\times _{\sigma }M$ be a Ricci soliton, then%
\begin{equation*}
\frac{1}{2}\left( \mathcal{\bar{L}}_{\bar{\zeta}}\bar{g}\right) \left( \bar{X%
},\bar{Y}\right) +\mathrm{\bar{R}ic}\left( \bar{X},\bar{Y}\right) =\lambda 
\bar{g}\left( \bar{X},\bar{Y}\right)
\end{equation*}%
where $\bar{X}=x\partial _{t}+X$ and $\bar{Y}=y\partial _{t}+Y$ are vector
fields on $\bar{M}$. Then%
\begin{eqnarray*}
-\lambda xyf^{2}+\lambda \sigma ^{2}g\left( X,Y\right) &=&\frac{1}{2}\left( 
\mathcal{\bar{L}}_{\bar{\zeta}}\bar{g}\right) \left( \bar{X},\bar{Y}\right) +%
\mathrm{\bar{R}ic}\left( \bar{X},\bar{Y}\right) \\
&=&\frac{1}{2}\left[ -2xyf^{2}\left[ \dot{h}+\zeta \left( \ln f\right) %
\right] +\sigma ^{2}\left( \mathcal{L}_{\zeta }g\right) (X,Y)+2h\sigma \dot{%
\sigma}g\left( X,Y\right) \right] \\
&&+\mathrm{\bar{R}ic}\left( x\partial _{t},y\partial _{t}\right) +\mathrm{%
\bar{R}ic}\left( x\partial _{t},Y\right) +\mathrm{\bar{R}ic}\left(
X,y\partial _{t}\right) +\mathrm{\bar{R}ic}\left( X,Y\right)
\end{eqnarray*}%
\begin{eqnarray*}
-\lambda xyf^{2}+\lambda \sigma ^{2}g\left( X,Y\right) &=&\frac{1}{2}\left[
-2xyf^{2}\left[ \dot{h}+\zeta \left( \ln f\right) \right] +\sigma ^{2}\left( 
\mathcal{L}_{\zeta }g\right) (X,Y)+2h\sigma \dot{\sigma}g\left( X,Y\right) %
\right] \\
&&+\left( n-1\right) \left( \frac{x\dot{\sigma}}{\sigma }\right) Y\left( \ln
f\right) +\left( n-1\right) \left( \frac{y\dot{\sigma}}{\sigma }\right)
X\left( \ln f\right) +\frac{n}{\sigma }xy\ddot{\sigma}+xy\frac{f^{\diamond }%
}{\sigma ^{2}} \\
&&+\mathrm{Ric}\left( X,Y\right) -\frac{1}{f}H^{f}\left( X,Y\right) -\frac{%
\sigma ^{\diamond }}{f^{2}}g\left( X,Y\right)
\end{eqnarray*}%
Let $X=Y=0$, we get%
\begin{eqnarray*}
xyf^{2}\left[ \dot{h}+\zeta \left( \ln f\right) \right] -\frac{n}{\sigma }xy%
\ddot{\sigma}-xy\frac{f^{\diamond }}{\sigma ^{2}}-\lambda f^{2}xy &=&0 \\
xy\left[ \dot{h}f^{2}+f\zeta \left( f\right) -\frac{n}{\sigma }\ddot{\sigma}-%
\frac{f^{\diamond }}{\sigma ^{2}}-\lambda f^{2}\right] &=&0
\end{eqnarray*}%
and so%
\begin{eqnarray*}
\dot{h}f^{2} &=&\lambda f^{2}-f\zeta \left( f\right) +\frac{n}{\sigma }\ddot{%
\sigma}+\frac{f^{\diamond }}{\sigma ^{2}} \\
\dot{h} &=&\frac{1}{f^{2}}\left( \lambda f^{2}-f\zeta \left( f\right) +\frac{%
n}{\sigma }\ddot{\sigma}+\frac{f^{\diamond }}{\sigma ^{2}}\right)
\end{eqnarray*}%
Now, let us put $x=y=0$, then%
\begin{eqnarray*}
\lambda \sigma ^{2}g\left( X,Y\right) &=&\frac{1}{2}\left[ \sigma ^{2}\left( 
\mathcal{L}_{\zeta }g\right) (X,Y)+2h\sigma \dot{\sigma}g\left( X,Y\right) %
\right] \\
&&+\mathrm{Ric}\left( X,Y\right) -\frac{1}{f}H^{f}\left( X,Y\right) -\frac{%
\sigma ^{\diamond }}{f^{2}}g\left( X,Y\right)
\end{eqnarray*}%
and so%
\begin{equation*}
\frac{1}{2}\sigma ^{2}\left( \mathcal{L}_{\zeta }g\right) (X,Y)+\mathrm{Ric}%
\left( X,Y\right) -\frac{1}{f}H^{f}\left( X,Y\right) =\left( \lambda \sigma
^{2}-h\sigma \dot{\sigma}+\frac{\sigma ^{\diamond }}{f^{2}}\right) g\left(
X,Y\right)
\end{equation*}
\end{proof}

The following corollaries are consequences of the above theorem.

\begin{corollary}
Let $\left( \bar{M},\bar{g},\bar{\zeta},\lambda \right) $ be a Ricci soliton
where $\bar{M}=_{f}I\times _{\sigma }M$ is a doubly warped spacetime\ and $%
\bar{\zeta}=h\partial _{t}+\zeta \in \mathfrak{X}\left( \bar{M}\right) $.
Then

\begin{enumerate}
\item $h\partial _{t}$ is a conformal vector field on $I$ with factor $\frac{%
2}{f^{2}}\left( \lambda f^{2}-f\zeta \left( f\right) +\frac{n}{\sigma }\ddot{%
\sigma}+\frac{f^{\diamond }}{\sigma ^{2}}\right) $.

\item $\left( M,g,\zeta ,\lambda \right) $ is a Ricci soliton if $f=\sigma
=1 $.

\item $\left( M,g,\zeta ,\lambda \right) $ is a Ricci soliton if $\sigma =1$
and $H^{f}=0$.
\end{enumerate}
\end{corollary}

\begin{theorem}
Let $\left( \bar{M},\bar{g},\bar{\zeta},\lambda \right) $ be a Ricci soliton
where $\bar{M}=_{f}I\times _{\sigma }M$ is a doubly warped spacetime\ and $%
\bar{\zeta}=h\partial _{t}+\zeta \in \mathfrak{X}\left( \bar{M}\right) $ is
a conformal vector field on $\bar{M}$ with factor $2\rho $. Then $\left(
M,g\right) $ is Einstein manifold with factor $\mu =\left( \lambda -\rho
\right) \sigma ^{2}+\frac{\sigma ^{\diamond }}{f^{2}}$ if $f$ is constant
\end{theorem}

\begin{proof}
Let $\left( \bar{M},\bar{g},\bar{\zeta},\bar{\lambda}\right) $ be a Ricci
soliton where $\bar{M}=_{f}I\times _{\sigma }M$ is a doubly warped
spacetime\ and $\bar{\zeta}=h\partial _{t}+\zeta \in \mathfrak{X}\left( \bar{%
M}\right) $ is a conformal vector field on $\bar{M}$. Then 
\begin{equation*}
\mathrm{\bar{R}ic}\left( \bar{X},\bar{Y}\right) =\left( \lambda -\rho
\right) \bar{g}\left( \bar{X},\bar{Y}\right)
\end{equation*}%
Let $x=y=0$, then%
\begin{equation*}
\mathrm{\bar{R}ic}\left( X,Y\right) =\left( \lambda -\rho \right) \sigma
^{2}g\left( X,Y\right)
\end{equation*}%
This equation implies that%
\begin{equation*}
\mathrm{Ric}\left( X,Y\right) -\frac{1}{f}H^{f}\left( X,Y\right) -\frac{%
\sigma ^{\diamond }}{f^{2}}g\left( X,Y\right) =\left( \lambda -\rho \right)
\sigma ^{2}g\left( X,Y\right)
\end{equation*}%
and so%
\begin{equation*}
\mathrm{Ric}\left( X,Y\right) -\frac{1}{f}H^{f}\left( X,Y\right) =\left[
\left( \lambda -\rho \right) \sigma ^{2}+\frac{\sigma ^{\diamond }}{f^{2}}%
\right] g\left( X,Y\right)
\end{equation*}
\end{proof}

\begin{corollary}
Let $\left( \bar{M},\bar{g},\bar{\zeta},\lambda \right) $ be a Ricci soliton
where $\bar{M}=_{f}I\times _{\sigma }M$ is a doubly warped spacetime\ and $%
\bar{\zeta}=h\partial _{t}+\zeta \in \mathfrak{X}\left( \bar{M}\right) $ is
a homothetic vector field on $\bar{M}$ with factor $2c$. Then 
\begin{equation*}
\lambda =c-\frac{1}{f^{2}}\left( \frac{n}{\sigma }\ddot{\sigma}+\frac{%
f^{\diamond }}{\sigma ^{2}}\right)
\end{equation*}
\end{corollary}

\begin{proof}
Let $\left( \bar{M},\bar{g},\bar{\zeta},\lambda \right) $ be a Ricci soliton
and $\bar{\zeta}=h\partial _{t}+\zeta \in \mathfrak{X}\left( \bar{M}\right) $
be a homothetic vector field on $\bar{M}$, then%
\begin{equation*}
\left( \lambda -c\right) \bar{g}\left( \bar{X},\bar{Y}\right) =\mathrm{\bar{R%
}ic}\left( \bar{X},\bar{Y}\right)
\end{equation*}%
for any vector fields $\bar{X}=x\partial _{t}+X$ and $\bar{Y}=y\partial
_{t}+Y$. Let us take $X=Y=0$, then%
\begin{eqnarray*}
\mathrm{\bar{R}ic}\left( x\partial _{t},y\partial _{t}\right) &=&\left(
\lambda -c\right) \bar{g}\left( x\partial _{t},y\partial _{t}\right) \\
xy\frac{f^{\diamond }}{\sigma ^{2}}+xy\frac{n}{\sigma }\ddot{\sigma}
&=&-xy\left( \lambda -c\right) f^{2} \\
xy\left( \frac{f^{\diamond }}{\sigma ^{2}}+\frac{n}{\sigma }\ddot{\sigma}%
+\left( \lambda -c\right) f^{2}\right) &=&0
\end{eqnarray*}%
Then%
\begin{equation*}
\lambda =c-\frac{1}{f^{2}}\left( \frac{n}{\sigma }\ddot{\sigma}+\frac{%
f^{\diamond }}{\sigma ^{2}}\right)
\end{equation*}%
and the proof is complete.
\end{proof}

\begin{theorem}
Let $\left( \bar{M},\bar{g},\bar{\zeta},\lambda \right) $ be a Ricci soliton
where $\bar{M}=_{f}I\times _{\sigma }M$ is a doubly warped spacetime\ and $%
\bar{\zeta}=h\partial _{t}+\zeta \in \mathfrak{X}\left( \bar{M}\right) $ is
a concurrent vector field on $\bar{M}$. Then

\begin{enumerate}
\item $\left( M,g\right) $ is Einstein manifold with factor $\mu =\left(
\lambda -2\right) \sigma ^{2}+\frac{\sigma ^{\diamond }}{f^{2}}$ if $f$ is
constant.

\item $\lambda =2-\frac{1}{f^{2}}\left( \frac{n}{\sigma }\ddot{\sigma}+\frac{%
f^{\diamond }}{\sigma ^{2}}\right) $.
\end{enumerate}
\end{theorem}

Let $\bar{\zeta}=h\partial _{t}+\zeta \in \mathfrak{X}\left( \bar{M}\right) $%
, then%
\begin{equation*}
\left( \mathcal{\bar{L}}_{\bar{\zeta}}\bar{g}\right) (\bar{X},\bar{Y}%
)=-2xyf^{2}\left[ \dot{h}+\zeta \left( \ln f\right) \right] +\sigma
^{2}\left( \mathcal{L}_{\zeta }g\right) (X,Y)+2h\sigma \dot{\sigma}g\left(
X,Y\right)
\end{equation*}%
and%
\begin{eqnarray*}
\mathrm{\bar{R}ic}\left( \bar{X},\bar{Y}\right) &=&xy\left( \frac{n\ddot{%
\sigma}}{\sigma }+\frac{f^{\diamond }}{\sigma ^{2}}\right) +\left(
n-1\right) \left( \frac{x\dot{\sigma}}{\sigma }Y\left( \ln f\right) +\frac{y%
\dot{\sigma}}{\sigma }X\left( \ln f\right) \right) \\
&&+\mathrm{Ric}\left( X,Y\right) -\frac{1}{f}H^{f}\left( X,Y\right) -\frac{%
\sigma ^{\diamond }}{f^{2}}g\left( X,Y\right)
\end{eqnarray*}%
Thus%
\begin{eqnarray}
&&\frac{1}{2}\left( \mathcal{\bar{L}}_{\bar{\zeta}}\bar{g}\right) (\bar{X},%
\bar{Y})+\mathrm{\bar{R}ic}\left( \bar{X},\bar{Y}\right)  \notag \\
&=&-xyf^{2}\left[ \dot{h}+\zeta \left( \ln f\right) \right] +\frac{1}{2}%
\sigma ^{2}\left( \mathcal{L}_{\zeta }g\right) (X,Y)+h\sigma \dot{\sigma}%
g\left( X,Y\right)  \notag \\
&&+xy\left( \frac{n\ddot{\sigma}}{\sigma }+\frac{f^{\diamond }}{\sigma ^{2}}%
\right) +\left( n-1\right) \left( \frac{x\dot{\sigma}}{\sigma }Y\left( \ln
f\right) +\frac{y\dot{\sigma}}{\sigma }X\left( \ln f\right) \right)  \notag
\\
&&+\mathrm{Ric}\left( X,Y\right) -\frac{1}{f}H^{f}\left( X,Y\right) -\frac{%
\sigma ^{\diamond }}{f^{2}}g\left( X,Y\right)  \label{er1}
\end{eqnarray}%
Suppose that both $f=\sigma =1$ are constants and $\left( M,g,\zeta ,\dot{h}%
\right) $ is a Ricci soliton on $M$, then%
\begin{eqnarray*}
&&\frac{1}{2}\left( \mathcal{\bar{L}}_{\bar{\zeta}}\bar{g}\right) (\bar{X},%
\bar{Y})+\mathrm{\bar{R}ic}\left( \bar{X},\bar{Y}\right) \\
&=&-xy\dot{h}+\frac{1}{2}\left( \mathcal{L}_{\zeta }g\right) (X,Y)+\mathrm{%
Ric}\left( X,Y\right) \\
&=&-xy\dot{h}+\dot{h}g(X,Y) \\
&=&\dot{h}\bar{g}(\bar{X},\bar{Y})
\end{eqnarray*}

Therefore, $\left( \bar{M},\bar{g},\bar{\zeta},\lambda \right) $ is a Ricci
soliton where $\lambda =\dot{h}$. This discussion leads us to the following
result.

\begin{theorem}
Let $\bar{M}=_{f}I\times _{\sigma }M$ be a doubly warped spacetime\ and $%
\bar{\zeta}=h\partial _{t}+\zeta \in \mathfrak{X}\left( \bar{M}\right) $ be
a vector field on $\bar{M}$. Then $\left( \bar{M},\bar{g},\bar{\zeta}%
,\lambda \right) $ is a Ricci soliton if

\begin{enumerate}
\item $\left( M,g,\zeta ,\dot{h}\right) $ is a Ricci soliton on $M$,

\item $f=\sigma =1$, and

\item $\lambda =\dot{h}$.
\end{enumerate}
\end{theorem}

Let $f=1$, $\zeta $ be a conformal vector field with factor $2\rho $ and $M$
be Einstein with factor $\mu $, then

\begin{eqnarray*}
&&\frac{1}{2}\left( \mathcal{\bar{L}}_{\bar{\zeta}}\bar{g}\right) (\bar{X},%
\bar{Y})+\mathrm{\bar{R}ic}\left( \bar{X},\bar{Y}\right) \\
&=&-xy\dot{h}+\rho \sigma ^{2}g(X,Y)+h\sigma \dot{\sigma}g\left( X,Y\right)
\\
&&+xy\left( \frac{n\ddot{\sigma}}{\sigma }\right) +\left( \mu -\sigma
^{\diamond }\right) g\left( X,Y\right) \\
&=&-xy\left( \dot{h}-\frac{n\ddot{\sigma}}{\sigma }\right) +\left( \frac{\mu
-\sigma ^{\diamond }}{\sigma ^{2}}+\rho +\frac{h}{\sigma }\dot{\sigma}%
\right) \sigma ^{2}g\left( X,Y\right)
\end{eqnarray*}%
i.e. $\left( \bar{M},\bar{g},\bar{\zeta},\lambda \right) $ is a Ricci
soliton if

\begin{eqnarray*}
\dot{h}-\frac{n\ddot{\sigma}}{\sigma } &=&\frac{\mu -\sigma ^{\diamond }}{%
\sigma ^{2}}+\rho +\frac{h}{\sigma }\dot{\sigma} \\
\left( \dot{h}-\rho \right) \sigma ^{2} &=&\mu +\left( n-1\right) \left(
\sigma \ddot{\sigma}-\dot{\sigma}^{2}\right) +h\sigma \dot{\sigma}
\end{eqnarray*}

\begin{theorem}
Let $\bar{M}=I_{f}\times _{\sigma }M$ be a doubly warped spacetime\ and $%
\bar{\zeta}=h\partial _{t}+\zeta \in \mathfrak{X}\left( \bar{M}\right) $ be
a vector field on $\bar{M}$. Then $\left( \bar{M},\bar{g},\bar{\zeta}%
,\lambda \right) $ is a Ricci soliton if

\begin{enumerate}
\item $\left( M,g\right) $ is a Einstein with factor $\mu $,

\item $f=1$, $\zeta $ is conformal with factor $2\rho $, and

\item $\left( \dot{h}-\rho \right) \sigma ^{2}=\mu +\left( n-1\right) \left(
\sigma \ddot{\sigma}-\dot{\sigma}^{2}\right) +h\sigma \dot{\sigma}$
\end{enumerate}

In this case $\lambda =\dot{h}-\frac{n\ddot{\sigma}}{\sigma }$.
\end{theorem}

\section{Acknowledgement}

We would like to thank the referee for the careful review and the valuable
comments, which provided insights that helped us to improve the quality of
the paper.

\appendix

\section{Concurrent vector fields on a doubly spacetime}

Let us now consider an example. Let $\bar{M}=_{f}I\times _{\sigma }%
%TCIMACRO{\U{211d} }%
%BeginExpansion
\mathbb{R}
%EndExpansion
$ be a $2-$dimension doubly warped spacetime equipped with the metric $\bar{g%
}=-f^{2}dt^{2}\oplus \sigma ^{2}dx^{2}$. Then%
\begin{eqnarray*}
\bar{D}_{\partial _{t}}\partial _{t} &=&\frac{ff^{\prime }}{\sigma ^{2}}%
\partial _{x} \\
\bar{D}_{\partial _{x}}\partial _{t} &=&\frac{f^{\prime }}{f}\partial _{t}+%
\frac{\dot{\sigma}}{\sigma }\partial _{x} \\
\bar{D}_{\partial _{t}}\partial _{x} &=&\bar{D}_{\partial _{x}}\partial _{t}
\\
\bar{D}_{\partial _{x}}\partial _{x} &=&-\frac{\sigma \dot{\sigma}}{f^{2}}%
\partial _{t}
\end{eqnarray*}%
A vector field $\zeta =h\partial _{t}+k\partial _{x}\in \mathfrak{X}(\bar{M}%
) $ is a concurrent vector field if%
\begin{eqnarray}
\bar{D}_{\partial _{t}}\zeta &=&\partial _{t}  \label{ew1} \\
\bar{D}_{\partial _{x}}\zeta &=&\partial _{x}  \label{ew2}
\end{eqnarray}%
The first equation implies that%
\begin{eqnarray*}
\bar{D}_{\partial _{t}}\left( h\partial _{t}+k\partial _{x}\right)
&=&\partial _{t} \\
\dot{h}\partial _{t}+\frac{hff^{\prime }}{\sigma ^{2}}\partial _{x}+k\left( 
\frac{f^{\prime }}{f}\partial _{t}+\frac{\dot{\sigma}}{\sigma }\partial
_{x}\right) &=&\partial _{t}
\end{eqnarray*}%
and so%
\begin{eqnarray}
\dot{h}f+kf^{\prime } &=&f  \label{ew3} \\
hff^{\prime }+k\sigma \dot{\sigma} &=&0  \label{ew4}
\end{eqnarray}%
Also, equation (\ref{ew2}) implies that%
\begin{eqnarray*}
\bar{D}_{\partial _{x}}\left( h\partial _{t}+k\partial _{x}\right)
&=&\partial _{x} \\
h\left( \frac{f^{\prime }}{f}\partial _{t}+\frac{\dot{\sigma}}{\sigma }%
\partial _{x}\right) +k^{\prime }\partial _{x}+k\left( -\frac{\sigma \dot{%
\sigma}}{f^{2}}\partial _{t}\right) &=&\partial _{x}
\end{eqnarray*}%
and so%
\begin{eqnarray}
hff^{\prime }-k\sigma \dot{\sigma} &=&0  \label{ew5} \\
h\dot{\sigma}+k^{\prime }\sigma &=&\sigma  \label{ew6}
\end{eqnarray}%
By solving equations (\ref{ew4}) and (\ref{ew5}), we get $hff^{\prime }=0$.
Thus $h=0$ or $f^{\prime }=0$. In both cases $k\sigma \dot{\sigma}=0$ i.e. $%
k=0$ or $\dot{\sigma}=0$. This discussion shows that we have the following
cases using equations (\ref{ew3}) and (\ref{ew6}):

\begin{case}
$h=0$ and $\dot{\sigma}=0$: then $kf^{\prime }=f$ and $k^{\prime }\sigma
=\sigma $ and so $k=x+a\neq 0$ and%
\begin{equation*}
\frac{f^{\prime }}{f}=\frac{1}{x+a}
\end{equation*}%
Therefore, $f=r\left( x+a\right) $ where both $r$ and $\left( x+a\right) $
are positive.
\end{case}

\begin{case}
$f^{\prime }=0$ and $k=0$: then $\dot{h}f=f$ and $h\dot{\sigma}=\sigma $ and
so $h=t+a\neq 0$ and similarly where both $r$ and $\left( t+a\right) $ are
positive.
\end{case}

\begin{case}
$f^{\prime }=0$ and $\dot{\sigma}=0$: then $\dot{h}f=f$ and $k^{\prime
}\sigma =\sigma $ and so $h=t+a$ and $k=x+b$.
\end{case}

The following table summarizes the above three cases of concurrent vector
fields on the $2-$dimensional doubly warped spacetime.

\begin{center}
\begin{tabular}{|c|c|c|c|c|}
\hline
\multicolumn{2}{|c|}{Case} & $\zeta $ & $\sigma $ & $f$ \\ \hline
$h=0$ & $\dot{\sigma}=0$ & $\zeta =\left( x+a\right) \partial _{x}$ & 
constant & $r\left( x+a\right) $ \\ \hline
$k=0$ & $f^{\prime }=0$ & $\zeta =\left( t+a\right) \partial _{t}$ & $%
r\left( t+a\right) $ & constant \\ \hline
$f^{\prime }=0$ & $\dot{\sigma}=0$ & $\zeta =\left( t+a\right) \partial
_{t}+\left( x+a\right) \partial _{x}$ & constant & constant \\ \hline
\end{tabular}
\end{center}


\begin{thebibliography}{99}
\bibitem{Agaoka:1998} Yoshio Agaoka, In-Bae Kim, Byung Hak Kim and Dae Jin
Yeom, \emph{On doubly warped product manifolds}, Mem. Fac. lntegrated Arts
ttnd Sci., Hiroshlma Univ., Ser.IV, \textbf{24(}1998), pp. 1-10.

\bibitem{Allison1988} D.E. Allison, \emph{Geodesic completeness in static
spacetimes}, Geom. Dedicata \textbf{26 }(1988) 85--97.

\bibitem{Allison1991} D.E. Allison, \emph{Pseudocovexity in Lorentzian
doubly warped products}, Geom. Dedicata \textbf{39} (1991) 223--227.

\bibitem{Apostolopoulos:2005} P S Apostolopoulos and J G Carot, \emph{%
Conformal symmetries in warped manifolds}, Journal of Physics: Conference
Series 8 (2005) 28--33

\bibitem{Barros:2012} A. Barros and E. Ribeiro Jr., \emph{Some
characterizations for compact almost Ricci solitons}, Proc. Amer. Math. Soc. 
\textbf{140}(2012), 1033-1040.

\bibitem{Barros:2013} A. Barros, Jos\'{e} N. Gomes, E. Ribeiro Jr. \emph{A
note on rigidity of the almost Ricci soliton, }Archiv der Mathematik, 
\textbf{100(}2013), no. 5, 481-490.

\bibitem{Beem1982} J.K. Beem and T.G. Powell, \emph{Geodesic completeness
and maximality in Lorentzian warped products}, Tensor(N.S.) \textbf{39}
(1982) 31--36.

\bibitem{Berestovskii2008} V. N. Berestovskii, Yu. G. Nikonorov, \emph{%
Killing vector fields of constant length on Riemannian manifolds}, Siberian
Mathematical Journal, \textbf{49}(2008), Issue 3 , pp 395-407.

\bibitem{Bergh:2012} N Van den Bergh, \emph{Conformally }$2+2-$\emph{%
decomposable perfect fluid spacetimes with constant curvature factor
manifolds}, Class. Quantum Grav. \textbf{29}(2012) 235003 (9pp)

\bibitem{Bishop1969} R. L. Bishop and B. O'Neill, \emph{Manifolds of
negative curvature}, Trans. Amer. Math. Soc. \textbf{145} (1969), 1-49.

\bibitem{Carot:1994} J. Carot, J. da Costa, and E. G. L. R. Vaz, \emph{%
Matter collineations: The inverse \textquotedblleft symmetry
inheritance\textquotedblright\ problem}, J. Math. Phys. \textbf{35}(1994),
4832.

\bibitem{Carot:2002} Jaume Carot and Brian O J Tupper, \emph{Conformally
reducible 2 + 2 spacetimes}, Classical and Quantum Gravity, \textbf{19}%
(2002), 4141--4166.

\bibitem{Carot:2008} Jaume Carot, Aidan J Keane and Brian O J Tupper, \emph{%
Conformally reducible 1+3 spacetime}, Class. Quantum Grav. \textbf{25}(2008)
055002 (43pp).

\bibitem{Chen2015} B. Y. Chen and S. Deshmukh, \emph{Ricci solitons and
concurrent vector fields}, Balkan Journal of Geometry and Its Applications, 
\textbf{20}(2015), no.1, , pp. 14-25.

\bibitem{Deshmokh20141} S. Deshmukh and F. R. Al-Solamy, \emph{Conformal
vector fields on a Riemannian manifold, }Balkan Journal of Geometry and Its
Applications, \textbf{19}(2014), no.2, pp. 86-93.

\bibitem{Deshmokh20142} S. Deshmukh and F. R. Al-Solamy, \emph{A note on
conformal vector fields on a Riemannian manifold, }Colloq. Math. \textbf{136}
(2014), 65-73.

\bibitem{Unal2012} F. Dobarro, B. Unal, \emph{Characterizing killing vector
fields of standard static spacetimes}, J. Geom. Phys. \textbf{62} (2012),
1070--1087.

\bibitem{Faghfouri:2015} M. Faghfouri and A. Majidi, \emph{On doubly warped
product immersion}, Journal of Geometry,\emph{\ }\textbf{106}(2015), no. 2,
pp 243-254.

\bibitem{Fernandez:2011} M. Fern\'{a}ndez-L\'{o}pez, Eduardo Garc\'{\i}a-R%
\'{\i}o, \emph{Rigidity of shrinking Ricci solitons}, Mathematische
Zeitschrift, \textbf{269(}2011), Issue 1-2, pp 461-466.

\bibitem{Gebarowski1993} A. Gebarowski, \emph{Doubly warped products with
harmonic Weyl conformal curvature tensor}, Colloquium Mathematicum, LXVII
(1993) 73--89.

\bibitem{Gebarowski:1995} A Gebarowski, \emph{On conformally flat doubly
warped products}, Soochow Journal of Mathematics, \textbf{21}(1995), no.1,
pp 125-129.

\bibitem{Gebarowski1996} A. Gebarowski, \emph{On conformally recurrent
doubly warped products}, Tensor (N.S.) \textbf{57} (1996) 192--196.

\bibitem{Hall:2004} G.S. Hall, \emph{Symmetries and Curvature Structure in
General Relativity }World. Scientific, Singapore, 2004.

\bibitem{Kuhnel:1997} W. Kuhnel and H. Rademacher, \emph{Conformal vector
fields on pseudo-Riemannian spaces}, Journal of Geometry and its
Applications, \textbf{7}(1997), 237--250.

\bibitem{Munteanu:2013} Ovidiu Munteanu, Natasa Sesum, \emph{On Gradient
Ricci Solitons}, Journal of Geometric Analysis, \textbf{23(}2013), no. 2, pp
539-561.

\bibitem{Olteanu2010} A. Olteanu, \emph{A general inequality for doubly
warped product submanifolds}, Math. J. Okayama Univ. \textbf{52}(2010),
133--142.

\bibitem{Olteanu2014} A. Olteanu,Doubly warped products in S-space forms,
Romanian Journal of Mathematics and Computer Science, \textbf{4}(2014), no.
1, p.111-124.

\bibitem{Oneill1983} B. O'Neill, \emph{Semi-Riemannian Geometry with
Applications to Relativity}, Academic Press Limited, London, 1983.

\bibitem{Perktas:2010} Selcen Y\"{A}uksel Perktas and Erol Kilic, \emph{%
Biharmonic maps between doubly warped product manifolds}, Balkan Journal of
Geometry and Its Applications, \textbf{15}(2010), no.2, pp. 1591-170.

\bibitem{Peterson:2009} P. Petersen and W. Wylie, \emph{Rigidity of gradient
Ricci solitons}, Pacific Journal of Mathematics, \textbf{241}(2009), no. 2,
329-345.

\bibitem{Ramos:2003} M. P. M. Ramosa, E. G. L. R. Vaz and J. Carot, \emph{%
Double warped space--times}, Journal of Mathematical Physics, \textbf{44}%
(2003), no. 10, pp. 4839-4865.

\bibitem{Sanchez1999} M. S\'{a}nchez, \emph{On the Geometry of Generalized
Robertson-Walker Spacetimes: Curvature and Killing fields}, J. Geom. Phys., 
\textbf{31} (1999), No.1, 1-15.

\bibitem{Shenawy:2015} S. Shenawy and B. Unal. $2-$Killing vector fields on
warped product manifolds, International Journal of Mathematics, \textbf{26}%
(2015), 17 pages.

\bibitem{Steller2006} M. Steller, \emph{Conformal vector fields on
spacetimes, }Ann Glob Anal Geom \textbf{29}(2006), 293--317.

\bibitem{Unal:PhD} Bulent Unal, \emph{Doubly warped products}. Ph.D. Thesis,
University of Missouri, Columbia (2000).

\bibitem{Unal:2001} Bulent Unal, \emph{Doubly warped products}, Differential
Geometry and its Applications \textbf{15}(2001) 253--263.
\end{thebibliography}
\end{document}